\documentclass[11pt,leqno,oneside]{amsart}
\usepackage[utf8]{inputenc}
\usepackage[T1]{fontenc}
\usepackage{amsmath,amsthm}
\usepackage{amssymb}
\usepackage{xcolor}
\usepackage[english]{babel}

\usepackage[left=2cm,top=2.5cm,bottom=2.5cm,right=2cm]{geometry}
\usepackage[colorlinks,citecolor=blue]{hyperref}
\everymath{\displaystyle}

\newtheorem{thm}{Theorem}
\newtheorem{prop}{Proposition}
\newtheorem{lemma}{Lemma}
\newtheorem{cor}{Corollary}
\theoremstyle{remark}
\newtheorem{ex}{Example}
\newtheorem{rem}{Remark}

\newcommand{\al}{\alpha}

\title{Billiards with singular invariant curves}
\author{Stefano Baranzini}
\email{stefano.baranzini@unito.it}
\thanks{The author is partially supported by the 2025 INDAM-GNAMPA project "Non integrabilità e complessità in Meccanica Celeste" and by the 2022 PRIN project "Stability in Hamiltonian Mechanics and Beyond"}
\address{University of Turin, Department of Mathematics, via Carlo Alberto 10, Turin, Italy}


\makeatletter
\@namedef{subjclassname@2020}{\textup{2020} Mathematics Subject Classification}
\makeatother


\date{}

\begin{document}
	\begin{abstract}
		We investigate the regularity of invariant curves of rotation number $1/2$ for a special class of symplectic twist maps of the annulus, billiard maps. We construct strictly convex smooth tables close to the circle having singular (i.e. not $C^1$) invariant curves. Our method relies on a modification of the classical string construction and allows precise control over the location of singularities: they form a discrete set whose closure can contain virtually any closed subset of $\mathbb{S}^1$. Each singularity corresponds to a hyperbolic $2$-periodic trajectory and the invariant curves admit distinct one-sided derivatives at these points.  An analogous construction yields perturbations of constant-width tables with invariant curves of rotation number $1/2$.
	\end{abstract}
	\maketitle
    \section{Introduction}
    Twist maps of the annulus were introduced as the local generic prototype of a symplectic map on a surface near an elliptic fixed point and as such, they have been thoroughly studied. In particular, finding invariant sets and understanding their regularity is a problem of the utmost importance. For instance, homotopically non-trivial invariant curves (also called \emph{essential}) bound regions invariant for the dynamics. 
    
    A celebrated theorem of Birkhoff (see \cite{birkhoff_invariant}) states that essential invariant curves of a $C^1$ symplectic twist map $f$ of the annulus are Lipschitz continuous.
    Simple examples show that, in general, invariant curves cannot be expected to be as regular as the map $f$. It is rather straightforward to produce twist maps of the annulus having a hyperbolic periodic point with coincident stable/unstable manifolds, giving an invariant curve that is not differentiable (see \cite{arnaud_three_results} and Example \ref{example:periodic_points} below).
    
    The situation for non-rational rotation numbers is much more delicate. It is a problem posed by J. Mather (see \cite{notes_cetraro}) to find non-differentiable curves containing no periodic points for $C^r$ twist maps of the annulus. For $r=1$ and $2$ an affirmative answer to this question are given in \cite{arnaud_non_diff_curve,arnaud_C2,avila_fayad_2022}. 
    
    In the present paper, we focus on a very special class of twist maps, \emph{billiard maps}. We propose a way to produce singular (i.e. not $C^1$) invariant curves of rotation number $\frac{1}{2}$ for tables close to the circle. This is achieved using a modification of the classical string construction introduced in \cite{birkhoff_kepler}. The essential curves we build are smooth outside a discrete set $S$ whose closure can contain essentially any closed subset of the boundary, modulo the symmetry imposed by the periodicity (see Theorem \ref{thm:main}). Moreover, each singular point $s \in S$ corresponds to a hyperbolic 2-periodic trajectory, and the invariant curves admit right and left derivatives there.
    
    The structure of the paper is the following. In Section \ref{sec:example} we discuss and slightly modify the example in \cite{arnaud_three_results}. In Section \ref{sec:main_result} we discuss the main results of the paper. Sections \ref{sec:periodic} and \ref{sec:perturbation_circle} are devoted to the proof of Theorem \ref{thm:periodic} and Theorem \ref{thm:main} respectively.

    \subsection{A motivating example}
    \label{sec:example}  
    The first point we briefly discuss is a modification of the example in \cite{arnaud_three_results}. Instead of considering a map $f$ with a single periodic hyperbolic point with coincident stable/unstable manifolds, we add several of those points having a chain of heteroclinic connections. This introduces some ambiguity in the definition of \emph{the} invariant essential curve since there are actually \emph{many} distinct ones, indeed, no smooth branch of the invariant curve covers the whole base under the action of $f$. With a slight abuse of language we will nevertheless speak of \emph{the} invariant curve to refer to the particular one we will consider. 
    
    The observation we make is that the set of singular points of an invariant essential curve is rather simple, it is just made up of isolated points.  Globally, however, this does not have be the case and singular points can accumulate on virtually any closed subset of the domain. This is precisely the same behaviour observed in Theorem \ref{thm:main} for the $\frac12$-invariant curve for billiards.
     
    \begin{ex}
    	\label{example:periodic_points}
    	Let $a,b \in \mathbb{Z}^*$ and consider a function $V:\mathbb R \to \mathbb R$ which is $1-$periodic. We consider the time 1 map given by the Hamiltonian 
    	\[
    	H(p,x,t) = \frac{1}{2}p^2+V(bx-at).
    	\]
    	It is straightforward to check that the change of coordinates 
    	\[
    	X = bx-at, \quad P = bp-a
    	\]
    	transforms the flow given by $H$ into an autonomous one, given by the Hamiltonian
    	\[
    	K(P,X) =  \frac{1}{2}P^2+b^2V(X).
    	\]
    	Thus, any point in the set $\{s:V'(s) = 0\}$ corresponds to a fixed point for the flow generated by $K$ and to a $\frac{b}{a}-$periodic point for the time 1 map of $H$. Moreover, if any such point $s_0$ is a non-degenerate maximum for $V$, then $s_0$ is hyperbolic. Thus, we can choose the portion of the stable/unstable manifolds of $s_0$ having positive momentum. This gives a non-differentiable invariant curve for $K$, and thus for $H$, at any hyperbolic point. However, if $V''(s) =0$ instead, the invariant curve is $C^1$ as a simple computation using the explicit parametrization given by the conservation of $K$ shows.
    	
    	We can now take a closed set $C \subset [0,1]$ having empty interior. Its complement can be written as a countable union of open intervals $I_n = (a_n,b_n)$. We can assume that $a_n\ne0$ and that $b_n\ne 1$ since the cases $(0,c)$ and $(c',1)$ can be treated similarly by considering $(c',c+1)$. For any such interval we can choose a monotone sequence $x_k^n$ such that $x_k^n \to a_n$ as $k\to -\infty$ and $x_k^n \to b_n$ as $k\to \infty$. We can find a function $f_n$ on $[0,1]$ which is smooth and whose support is $[a_n,b_n]$. Moreover we require that $f_n(x_k^n) = 0$, $f_n'(x_k^n)>0$ and $\int_{x_k^n}^{x_{k+1}^n}f_n(s) ds =0$ for all $k$. Define then $f(s) = \sum_{n \in \mathbb{N}} f_n(s)$, extend it by periodicity and set  $V(s) = \int_0^s f(s) ds$. Clearly, all $x_n^k$ are non-degenerate maxima, $V(x_n^k)$ is the same for all $n$ and $k$ and they  accumulate on $C$. 
    \end{ex}
    
    \subsection{Main result}
    \label{sec:main_result}  
    We consider a compact, strictly convex  set $\Omega \subset \mathbb{R}^2$ with smooth boundary and the cylinder $\mathcal{C}$ defined as
    \[
    \mathcal{C} = \{ (p,v_p) : p \in \partial \Omega, \text{ and } v_p \text{ is an inward pointing vector at } p\}.
    \]
    The billiard map $T: \mathcal C \to \mathcal C$ sends a pair  $(p,v_p)$ to $(p',v'_{p'})$ if $p'\ne p$ belongs to the intersection of $\partial \Omega$ and the line through $p$ with direction $v_p$. The new vector $v'_{p'}$ is then the reflection of $v_p$ about the tangent line to $\partial \Omega$ at $p'$.    
    It turns out that $T$ is an area preserving twist map of the annulus (see \cite[Theorem 3.1]{tabachnikov_billiards}). 
    
    We call $\Omega$ a \emph{constant width-body} if the corresponding billiard map $T$ has an invariant curve of rotation number $\frac12$, foliated by periodic points. In other words, $\Omega$ is of constant width if all lines orthogonal to one point of the boundary are orthogonal to the other intersection point as well.
    
    \begin{rem}
    	 In general, billiard maps have \emph{many} invariant curves near the boundary of $\mathcal{C}$. This follows from KAM theory and the results in \cite{lazutkin}. It should be noted that all these curves are not just Lipschitz but smooth.
    \end{rem}
   
    Throughout the paper we parametrize $\partial \Omega$ by its tangent line. We fix a parametrization $\gamma$ of the form
    \[
    \gamma(t) = \gamma(0) +\int_0^t \rho(s) e^{is}ds, 
    \]
    where $\rho(s)$ denotes the curvature radius of $\gamma$.
    This prompts us to identify $2-$periodic orbits with the direction to which they are orthogonal to.  It is thus convenient to consider the projective line $\mathbb{P}^1$ and its natural quotient map $\Pi: \mathbb{S}^1 \to \mathbb{P}^1$.     
    In all our statements, we will fix a closed subset $\tilde C\subseteq\mathbb{P}^1$ and consider its pre-image $C = \Pi^{-1}(\tilde{C})$  in $\mathbb{S}^1$. The vectors in  $\{ie^{is}:s \in C\}$  (i.e.  the direction in $i \tilde C$) will represent the directions spanning a $2-$periodic trajectory.
    
    We can now state the main results of the paper. 
    \begin{thm}
    	\label{thm:periodic}
    	Let $\Omega$ be a constant width-body. 
    	For any closed set $\tilde{C} \subset \mathbb{P}^1$ there exists smooth $1-$parameter families of smooth, strictly convex tables $\Omega_\tau$ such that
    	\begin{itemize}
    		\item $\Omega_0 = \Omega$,
    		\item $\Omega_\tau$ has two $T$-invariant curves $\delta_\tau$ and $\delta^-_\tau$ of rotation number $\frac12$,
    		\item $\Omega_\tau$ has periodic trajectories exactly in the directions given by $i \,\tilde{C}$.
    	\end{itemize}
    \end{thm}
     
     We will denote the set of singular points of $\delta_\tau$ (and of $\delta_\tau^-$) by $S_\tau$. Points in $S_\tau$ are expected to correspond to hyperbolic $2-$periodic trajectories (compare with Remark \ref{rem:singularities} and Corollary \ref{cor:hyperbolicity}). We will check the hyperbolicity condition for perturbations of the circle, yielding the following result.
	\begin{thm}
	 	\label{thm:main}
	 	For any non-empty closed set $\tilde{C} \subset \mathbb{P}^1$ with empty interior there exists smooth $1-$parameter families of smooth, strictly convex tables $\Omega_\tau$ as in Theorem \ref{thm:periodic} such that additionally
	 	\begin{itemize}
	 		\item $\Omega_0$ is a circle,
	 		\item $S_\tau$ is a discrete set, 
	 		\item the closure $\overline{S_\tau}\supset \Pi^{-1}(\tilde{C})$.
	 	\end{itemize}	    
	\end{thm}
    
     \begin{rem}
     	All examples appearing in Theorem \ref{thm:main} are non-generic in the sense that generic perturbations of the circle will break the invariant curve of rotation number $\frac{1}{2}$ (compare with \cite[Theorem 1]{ramirez_ros}). Nevertheless, the families built in Theorem \ref{thm:periodic} and \ref{thm:main} are infinite dimensional, as it happens for tables admitting a rational caustics of a specific rotation number (see \cite{persistence_rational_caustics}). As it will be clear from the proofs in Section \ref{sec:periodic}  and \ref{sec:perturbation_circle},  they are essentially infinitesimally identified by the datum of a function vanishing on the set $\Pi^{-1}(\tilde{C})$.
     \end{rem}
     \begin{rem}
     	\label{rem:singularities}
     	It should be stressed that the construction giving the invariant curves in Example \ref{example:periodic_points} and in Theorem \ref{thm:main} cannot lead to any other type of singularity other than having different right and left derivative, on a discrete set. Indeed, the proof of Theorem \ref{thm:main} shows the existence of two smooth invariant curves for the square of the billiard map (compare with \ref{lemm:derivatives_invariant_curve} and Section \ref{sec:perturbation_circle}). Singular points correspond to transversal intersections of these two curves which necessarily are isolated. 
     \end{rem}

    \section{Proof of Theorem \ref{thm:periodic}}
    \label{sec:periodic}
    In this section we set up a variation of the classical string construction. We follow the lines of \cite[Section 4]{knill} and \cite[Section 4.1]{birkhoff_kepler}. Let us consider two smooth constant width bodies $\Omega_1$ and $\Omega_2$ whose boundary is parametrized by two curves $\gamma_1$ and $\gamma_2$. Denoting the curvature radii $\rho_1$ and $\rho_2$ respectively, we have the following parametrization of $\partial \Omega_1$ and $\partial \Omega_2$
	\begin{equation*}
		\gamma_j(t) = \gamma_j(0)+ \int_0^t \rho_j(s)e^{is} ds, \text{ for }j=1,2. 
	\end{equation*}
    Recall that (see also \cite[Proposition 4.1]{knill}) the functions $\rho_j>0$ can be represented in Fourier series as
    \[
       \rho_j(t) = \sum_{k \in \mathbb{Z}} a_k^j e^{i k t}
    \] 
    and the coefficients $a^j_k$ satisfy $a^j_0 \ne 0$, $a^j_{-1}=0$, $\overline{a^j_k}= a^j_{-k}$ and $a^j_{2k}=0$ for $j=1,2$. 
    To any constant width body, we can associate a set of coordinates in $\mathbb{R}^2\setminus \Omega_j$ writing
    \[
    (r_j,t) \mapsto \gamma_j(t)+r_jie^{it}, 
    \]
    where $r_j$ represent the distance from $\Omega_j$. 
    We now perform the same \emph{string construction} as in \cite[Section 4.1]{birkhoff_kepler} and consider, for $\ell>0$, the tables defined by the equation
    \begin{equation}
    \label{eq:level_set_distance}
        d(\Omega_1,x)+d(\Omega_2,x) = \ell .
    \end{equation}
    Any smooth level set of the two function above defines a table having an invariant curve of rotation number $1/2$. Moreover, each pencil of lines orthogonal to $\Omega_i$ gives rise to an invariant curve for the square of the billiard map which we denote by $\delta_\pm$.    
    These, in general, \emph{are not} composed by periodic points, but of a combination of periodic points and their heteroclinic connections.
    Periodic points are characterized by the fact that the two lines orthogonal to $\Omega_1$ and $\Omega_2$ respectively  and passing through a given point $x$ coincide. Using the parametrization above, this implies that for some $t\in [0,2 \pi]$
    \begin{equation}
    	\label{eq:find_zero}
    \gamma_1(t)-\gamma_2(t) = (r_2-r_1)i e^{it}.
    \end{equation}
    This mean that periodic points correspond to solutions of the equation
    \[
          \langle \gamma_1(t)-\gamma_2(t), e^{it} \rangle =0.
    \]
    We wish now to study the possible set of solutions to this equation. To simplify notation, it is convenient to introduce the following quantities
    \[
    c = (c_1,c_2) = \gamma_1(0)-\gamma_2(0), \quad f(t) = \rho_1(t)-\rho_2(t) = \sum_{k \in \mathbb{Z}} \al_k e^{i kt}
    \]
    where now $f$ leaves in the closed subspace  \[
    V = \{\al_1 =0,\al_{2k} = 0,\al_{k}=\overline{\al_{-k}} \text{ for }k\ge1\}.\]
    Set $\al_k = a_k+ib_k$, plugging in the Fourier series, we can rewrite equation \eqref{eq:find_zero} in the following way
    \begin{equation}
    	\label{eq:find_zero_fourier}
    	\begin{aligned}
      	 \langle c + \int_0^t f(s)e^{i s}ds, e^{i t}\rangle &=\langle c + \int_0^t \sum_{k \in \mathbb{Z}} \al_k e^{i (k+1)s}ds, e^{i t}\rangle\\
      	 &= \langle c,e^{it}\rangle +\sum_{k\in \mathbb{Z}} \langle\frac{\al_k(e^{i(k+1)t}-1)}{i(k+1)}, e^{it} \rangle  \\
      	 & = \langle c+i\sum_{k\in \mathbb{Z}}\frac{\alpha_k}{k+1} ,e^{it}\rangle +\sum_{k\in \mathbb{Z}} \langle\frac{(b_k-ia_k)e^{i(k+1)t}}{k+1}, e^{it} \rangle  \\
      	 & = \gamma_1 \cos(t)+\gamma_2 \sin(t)  +\sum_{k>0} \frac{2k}{k^2-1} \left(a_k\sin(kt)+b_k\cos(kt)\right).
    	\end{aligned}
    \end{equation}
    where $\gamma_1 = c_1-\sum_{k\in \mathbb{Z}}\frac{b_k}{k+1}$
 and $\gamma_2 = c_2+\sum_{k\in \mathbb{Z}}\frac{a_k}{k+1}$. Since $\al_{2m} =0$ for all $m$, the function above has an expansion in odd frequencies alone.  Any function $g$ admitting such an expansion has the following symmetry
 \begin{equation}
 	\label{eq:symmetry_zero _set}
   g(t) = -g(t+\pi),\quad \forall t \in[0,2\pi].
 \end{equation}

\begin{proof}[Proof of Theorem \ref{thm:periodic}]
	To prove the statement, we need to show that for any set $\tilde{C}\subset \mathbb{P}^1$ we can build a function vanishing on $\Pi^{-1}(\tilde{C})$ and satisfying the symmetry condition in equation \eqref{eq:symmetry_zero _set}.  Let $c\in \Pi^{-1}(\tilde{C}) $ and set $C' = \Pi^{-1}(\tilde{C}) \cap [c,c+\pi]$. We can pick any smooth function with derivatives of all order vanishing at $c$ and $c+\pi$ which is positive on $[c,c+\pi]\setminus C'$ and then extended it to $\mathbb{S}^1$ using \eqref{eq:symmetry_zero _set}.
	
	Once this is settled, the statement follows from the construction outlined above. Let $\Omega$ a constant-width body with curvature radius $\rho$. Pick $\rho_1 = \rho$ and $\rho_2 = \rho +\tau f$ where $f\in V$ as above. Thanks to equation \eqref{eq:find_zero_fourier}, the Fourier coefficients of $f$ can be directly recovered from the ones of 
	\[
	g(t) = \langle c + \int_0^t f(s)e^{i s}ds, e^{i t}\rangle
	\] 
	and viceversa. Moreover, if $g$ is $C^\infty$ so is $f$. Clearly, for $\tau = 0$ we recover the original body $\Omega$. Recall that $\Omega_\tau$ has two invariant curve for the square of the billiard map denoted by
	\[
	\delta^\tau_\pm(t) = (t,\lambda_{\pm}^\tau(t)).
	\]	
	We define the invariant curve $\delta_\tau$  as
	\[
	\delta_\tau(t) = (t,\lambda_\tau(t)), \quad \lambda_\tau(t) = \max\{\lambda_{+}^\tau(t),\lambda_{-}^\tau(t)\}
	\]
	and $\delta_\tau^-$ as the minimum. They are invariant thanks to same argument of \cite[Lemma 4.6]{birkhoff_kepler}. The last point follows by construction.

\end{proof}

\section{Proof of Theorem \ref{thm:periodic}}
\label{sec:perturbation_circle}
In this section we consider perturbations of the circle and try to analyse which conditions guarantee that the invariant curves in Theorem \ref{thm:periodic} intersect transversally, giving rise to a invariant curve for $T$ with singular points.  Without loss of generality, we assume that $\gamma_1 (t) = -i e^{i t}$  is a circle and \[\gamma_2(t) = \gamma(t) = \gamma(0)+\int_0^t  \rho(s) e^{is} ds\]
 is a constant width-body close to a circle.

We follow again the exposition in \cite[Section 4.1]{birkhoff_kepler} and of Section \ref{sec:periodic}. We have an explicit parametrization $\Gamma$ of $ \partial \Omega$ which is the level set of the function in \eqref{eq:level_set_distance}. In this particular case, it is given by (see \cite[Eq. (21)]{birkhoff_kepler})
\[
\Gamma(t) = \gamma(t) - \frac{i}{2}\frac{\ell^2- \vert \gamma(t) \vert^2}{\ell- \langle \gamma(t), i e^{i t}\rangle }e^{i t}, \quad t \in [0,2 \pi].
\]
Let us observe that, by construction, the quantity
\begin{equation}
	\label{eq:def_s}
s(t) = \frac{1}{2}\frac{\ell^2- \vert \gamma(t) \vert^2}{\ell- \langle \gamma(t), i e^{i t}\rangle }
\end{equation}
coincides with $\ell-\vert \Gamma(t)\vert$ and determines the distance of $\Gamma(t)$ from the origin. 
As already mentioned, the construction of Section \ref{sec:periodic} proves the existence of invariant curves for the \emph{square} of the billiard map as well, which we can describe rather explicitly in our particular case. The respective trajectories belong either to the pencil of line through the origin (the circle caustic) or their reflections about the tangent line to the boundary. Moreover, recall that  periodic trajectories occur whenever a line through the origin meets $\partial \Omega$ orthogonally and correspond to critical points of the distance from the origin. We denote by $\delta_\pm$ the invariant curves for the square of the billiard map. We have an explicit parametrization for them  (see \cite[Lemma 4.7]{birkhoff_kepler}) as 
\[
\delta_{\pm}(t) = \left( t, \arccos\left(\pm \frac{\langle \Gamma(t),\dot{\Gamma}(t)\rangle}{\vert \Gamma(t)\vert \vert \dot{\Gamma}(t)\vert} \right)\right).
\]
As the next Lemma shows, the derivatives of the invariant curves at a crossing point are completely determine by the derivatives of $s(t)$. To compute them we consider for a moment a re-parametrization of $\Gamma$ using the arc-length parameter $\xi$.  We have the following Lemma.
\begin{lemma}
	\label{lemm:derivatives_invariant_curve}
	Let $\Gamma$ be parametrized using arc-length and $\xi_0$ be a critical point of the distance. The derivatives of the curves $\delta_{\pm}$ read
	\begin{equation*}
		\delta'_{\pm}(\xi_0) =\left(1, \mp \left( \frac{1}{\vert \Gamma(\xi_0)\vert}-k(\xi_0)\right)\right), \quad \frac{d^2}{d \xi^2}\vert \Gamma(\xi)\vert \Big\vert_{\xi=\xi_0}   = \frac{1}{\vert \Gamma(\xi_0)\vert}-k(\xi_0).
	\end{equation*}
\end{lemma} 

We now write $\gamma$ as a perturbation of the circle and rewrite the function $s(t)$ in terms of this perturbation. Set $\rho(t) = 1 + f(t)$.  We have
\begin{align*}
\gamma(t) &= \gamma(0)+ \int_0^{t} (1 +  f(s))e^{is} ds \\&= \gamma(0)-i(e^{it}-1)+ \int_0^t f(s)e^{is}ds
\\& = -i e^{it}+\left(c+\int_0^t f(s)e^{is}ds\right)\\&= -ie^{it} +  p(t).
\end{align*}
where we chose $\gamma(0) = c -i $. Plugging in this expression in \eqref{eq:def_s} we find
\[
s(t) =  \frac{1}{2}\frac{\ell^2-1 +2 \langle  i e^{i t},p(t) \rangle - \vert p(t) \vert^2}{\ell+1-  \langle p(t), i e^{i t}\rangle}.
\] 
Let us consider the functions $g(t) = \langle p(t),e^{it}\rangle$ and $h(t) = \langle p(t), i e^{it}\rangle$. Since $\dot{p}(t) = f(t)e^{it}$ we have
\[
\dot{h}(t) = \langle \dot{p}(t),ie^{i t}\rangle -\langle p(t),e^{it}\rangle = -g(t).
\]
Observing that $\vert p(t)\vert^2 = h(t)^2+g(t)^2$, we can rewrite the function $s(t)$ solely in terms of $h(t)$. Indeed we have
\begin{equation}
	\label{eq:s_in_terms_of_h}
s(t) = \frac{1}{2}\frac{\ell^2-1 +2 h(t) - h(t)^2-\dot{h}^2(t)}{\ell+1- h(t)}.
\end{equation}

Recall that the set $\left\{\dot h = 0\right\}$ completely determines the periodic trajectories of $\Omega$  (compare with \eqref{eq:find_zero}) and that the sole datum of $g$ completely determines the perturbation $p(t)$ (compare with \eqref{eq:find_zero_fourier})


Theorem \ref{thm:main} is a consequence of the following Proposition. Let $\tilde{C} \subset \mathbb{P}^1$ be a closed set and $C =\Pi^{-1}(\tilde C) \subset \mathbb{S}^1$.
 We write
\[
\tilde C = \tilde{C}_1 \sqcup \tilde{C}_2, \quad C_i = \Pi^{-1}(\tilde{C}_i)
\] 
where $\tilde{C}_1$ is the set of isolated points of $\tilde{C}$ and $\tilde{C}_2$ its complementary. We have the following  
\begin{prop}
	\label{prop:non_degenerate_int}	
	For any closed set  $\tilde{C} \subset \mathbb{P}^1$ there are smooth families of compact strictly convex set $\Omega_\tau$ with smooth boundary having two smooth invariant curves for the square of the billiard map $\delta_\pm$ such that
	\begin{itemize}
		\item $\Omega_0$ is a circle
		\item all points in  $C$ correspond to periodic trajectories 
		\item $\delta_{\pm}$  meet transversally at all points in $C_1$.
	\end{itemize}
\begin{proof}
	
	The existence of families of such $\Omega_\tau$ follows from the construction outlined in Section \eqref{sec:periodic} with $f$ replaced by $\tau f$. 
	
	It remains to check the transversality condition on $\delta_\pm$. 	Thanks to Lemma \ref{lemm:derivatives_invariant_curve} the invariant curves $\delta_{\pm}$ meet transversally if and only if the second derivative of $s(t)$ is non-zero. Note that this, since we are always working at a critical point of $s(t)$, does not depend on the parametrization. 
		
	We can compute the second derivative of $s(t)$ using \eqref{eq:s_in_terms_of_h}. The general expression is slightly cumbersome, however, whenever $\dot h(t)$ is equal to zero i.e when $\dot s (t) =0$, it simplifies to
	\[
	\ddot{s}(t)\vert_{\{h'(t)=0\}} \sim \frac{\ddot{h}(t)}{2}\left( 1-\frac{2 \ddot{h}(t)}{\ell+1-h(t)}\right)
	\]
	If we fix $h(t)$ and the length of the string $\ell$ is big enough, uniformly in the $C^2$ norm of $h$, the second derivative of $s(t)$ is non-zero if and only if $\ddot{h}(t)$ (i.e. $\dot g$) is non-zero.	
    Thus, we have shown that the transversal intersection of $\delta_\pm$ at a critical point being non-zero is a condition on the function $g$ alone.

	It remains to check that for any closed set $C\subseteq \mathbb{S}^1$ we can find a smooth function $g$ having the symmetry in \eqref{eq:symmetry_zero _set}, vanishing on $C$ but with non zero derivative on the points belonging to $C_1$. This can be achieved as follow. 
	Let us consider the complement of $C$, it is the union of countable many intervals $\{I_k\}_{k \in \mathbb
	N}$. We can consider on this  the equivalence relation generated by
	\[
	   I_k \sim I_m \iff \bar{I}_k \cap \bar{I}_m \ne \emptyset.
	\] 
	Any equivalence class $\iota$ correspond to a \emph{maximal} sequence of discrete points of $C$ and the union of the closure of its elements is an interval whose endpoint cannot lie in $C_1$, i.e. it satisfies the following 
	\[
	    I_\iota =\bigcup_{I \in \iota} \bar{I}, \quad  \bar{I}_\iota= [a_\iota,b_\iota] \quad \{a_\iota,b_\iota\}\cap C\subset C_2.
	\]
	
	There are thus two possibilities depending on the cardinality of $C_1\cap I_{\iota}$ and we consider the case  $\vert C_1\cap I_{\iota} \vert =\infty$ since the other is straightforward.
	Let us consider a smooth function $\psi:\mathbb{R}\to \mathbb{R}$ such that $\psi \equiv 0$ on $(-\infty,-1]$ and $\psi \equiv 1$ on $[0,\infty)$. Let us fix $a,b>0$ and consider the function $\psi_a^b(t)$ defined as
	\[
	\psi_a^b(t) =  e^{-\frac{1}{a}-\frac{1}{b}} \psi\left(\frac{t}{a}\right) \psi\left(-\frac{t}{b}\right) \arctan t.
	\]
	Clearly, $\{\psi_a^b= 0\} = (-a,b)^c \cup \{0\}$ and $(\psi_a^b)'(0)\ne 0$.
	Let us pick $I_0\in I_\iota$ and order the elements of $I_\iota$ as $\{I_k\}_{k \in \mathbb
	Z}$ with respect to usual ordering of real numbers. We denote $I_k$ as $[a_k,b_k]$. Note that $b_k = a_{k+1}$ for all $k$. We define
	\[
       \psi_\iota(t) = \sum_{k \in \mathbb{Z}}(-1)^k \psi_{b_{k-1}-a_{k-1}}^{b_{k}-a_k}(t-a_k).
	\]
 	Clearly, $\psi\vert_{I_{\iota}}$ vanishes only on $C_1 \cap I_{\iota}$ and has non-zero derivative there. Moreover $\psi_{\iota}$ is identically zero on $\mathbb{S}^1\setminus I_\iota$. We can then define $\varphi$ 
 	\[
 	 \varphi(t) = \sum_{\iota}
 		\psi_{\iota}(t).
 	\]
 	which has the required properties.
\end{proof}
\end{prop}

\begin{proof}[Proof of Theorem \ref{thm:main}]
		The proof is an application of Proposition \ref{prop:non_degenerate_int}. Consider the complementary of $C$ and write it as a union of open intervals $\bigcup_{k\in \mathbb{Z}} J_k$. To each interval, remove a  monotone sequence $\{c^k_m\}_{m \in \mathbb{Z}}$, converging to the extrema of $I_k$ as $m \to \pm \infty$. 
		We obtain a new closed set  $C' = \bigcup_{k,m\in \mathbb{Z}}\{c_m^k\}\cup C$ having all the $c_{k,m}$ as isolated points. By assumption, $C =\partial C \ne \emptyset$ and any point in $\partial C$ can be approximated by a suitable subsequence of the $\{c_m^k\}_{m,k \in \mathbb {Z}}$ and the proof is concluded.
	\end{proof}
    
    Next, we show that all points in $C_1$ are hyperbolic. It should be noted that in general they will not be \emph{uniformly hyperbolic} since the eigenvalues of $dT^2$ will tend to $1$ as soon as we approach the region $\left\{\ddot h=0\right\}$.

    \begin{cor}
	\label{cor:hyperbolicity}
	Under the same hypothesis of Proposition \ref{prop:non_degenerate_int}, all points in $C_1$ are hyperbolic periodic points for the associated billiard map.
	\begin{proof}
		Let us observe that the curves $\delta_{\pm}$ are smooth and can be defined locally, near any point in $C_1$ as the zero set of a smooth function $F_\pm$. Since $\delta_\pm$ are invariant with respect to the square of the billiard map $T$  we have  
		\[
		\frac{d}{dt}\Big \vert_{t= t_0} F_{\pm}(T^2(\delta_{\pm}(t))) =\langle\dot\delta^\perp_\pm(t_0), d T^2 \dot\delta_\pm(t_0) \rangle =0.
		\]
		Since we know that $\dot{\delta}^\pm(t_0)$ are two linearly independent vectors, thanks to Lemma \ref{lemm:derivatives_invariant_curve} and Proposition \ref{prop:non_degenerate_int}, $dT^2$ has two eigenvectors. Thus, $d T^2$ can be either degenerate or hyperbolic. 
		
		The explicit form of $d T^2$ can be easily computed using the generating function of the billiard map (see \cite[Lemma 4.8]{birkhoff_kepler}) and has trace equal to $\pm 2$ if and only if one of the following holds
		\[
		k_1k_2 d - (k_1+k_2) =0, \quad d^2 k_1k_2- d(k_1+k_2)-2 =0,
		\]
		where $d$ stands for the chord's length. This equation can be satisfied only by the following values of $d$
		\[
		   \frac{k_1+k_2}{k_1k_2}, \frac{1}{k_1}, \frac{1}{k_2}.
		\]
		The values $1/k_1$ and $1/k_2$ are incompatible with fact that $\Omega$ is close to a circle. For the remaining one we can compute the curvature and $d$ at points $t$ and $t+\pi$ such that $\dot{h}(t) =0$ (recall $h(t) = -h(t+\pi)$). We have $d =1+l$ whereas
		\begin{align*}
		k_1 = \frac{1}{1 + l + h(t)} + \frac{1}{1 + l + h(t) + 2 \ddot{h}(t)}\\
		k_2 = \frac{1}{1 + l - h(t)} + \frac{1}{1 + l - h(t) - 2 \ddot{h}(t)}.
		\end{align*}
		A direct computations shows that
		\[
			k_1k_2 d - (k_1+k_2) = \frac{4 d \, \ddot{h}(t) }{(d^2 - h^2(t)) (d^2 - (h(t) + 2 \ddot{h}(t))^2)}
		\]
		which is non-zero when $\ddot{h}(t) \ne0$ concluding the proof.
	\end{proof}
\end{cor}
\bibliography{ref}

\begin{thebibliography}{10}

\bibitem{arnaud_three_results}
M.-C. Arnaud.
\newblock Three results on the regularity of the curves that are invariant by
  an exact symplectic twist map.
\newblock {\em Publ. Math. IHES}, 109:1--17, 2009.

\bibitem{arnaud_non_diff_curve}
M.-C. Arnaud.
\newblock A nondifferentiable essential irrational invariant curve for a
  ${C}^1$ symplectic twist map.
\newblock {\em Journal of Modern Dynamics}, 5(3):583--591, 2011.

\bibitem{arnaud_C2}
M.-C. Arnaud.
\newblock Boundaries of instability zones for symplectic twist maps.
\newblock {\em Journal of the Institute of Mathematics of Jussieu},
  13(1):19–41, 2014.

\bibitem{avila_fayad_2022}
A.~Avila and B.~Fayad.
\newblock Non-differentiable irrational curves for ${C}^1$ twist map.
\newblock {\em Ergodic Theory and Dynamical Systems}, 42(2):491–499, 2022.

\bibitem{birkhoff_kepler}
S.~Baranzini, V.~L.~Barutello, I.~De~Blasi, and S.~Terracini.
\newblock On the {B}irkhoff conjecture for {K}epler billiards, 2025.

\bibitem{persistence_rational_caustics}
Y.~Baryshnikov and V.~Zharnitsky.
\newblock Sub-riemannian geometry and periodic orbits in classical billiards.
\newblock {\em Mathematical Research Letters}, 13(4):587--598, 2006.

\bibitem{birkhoff_invariant}
G.~D. Birkhoff.
\newblock Surface transformations and their dynamical application.
\newblock {\em Acta Math.}, 43:1--119, 1920.

\bibitem{notes_cetraro}
H.~Eliasson, S.~Marmi, S.~Kuksin, and J.C. Yoccoz.
\newblock {\em Dynamical Systems and Small Divisors: Lectures given at the
  {C.I.M.E}. Summer School held in {C}etraro {I}taly, June 13-20, 1998}.
\newblock Lecture Notes in Mathematics. Springer Berlin Heidelberg, 2004.

\bibitem{knill}
O.~Knill.
\newblock On nonconvex caustics of convex billiards.
\newblock {\em Elem. Math}, 53:89--106, 1998.

\bibitem{lazutkin}
V.~F. Lazutkin.
\newblock The existence of caustics for a billiard problem in a convex domain.
\newblock {\em Math. USSR-Izv.}, 7:185--214, 1973.

\bibitem{ramirez_ros}
R.~Ramírez-Ros.
\newblock Break-up of resonant invariant curves in billiards and dual billiards
  associated to perturbed circular tables.
\newblock {\em Physica D: Nonlinear Phenomena}, 214(1):78--87, 2006.

\bibitem{tabachnikov_billiards}
S.~Tabachnikov.
\newblock {\em Geometry and Billiards}, volume~30.
\newblock American Mathematical Society, 2005.

\end{thebibliography}
\bibliographystyle{plain}
\end{document}